\documentclass[sn-mathphys-num]{sn-jnl}% Math and Physical Sciences Numbered Reference Style 
\usepackage{graphicx}%
\usepackage{multirow}%
\usepackage{amsmath,amssymb,amsfonts}%
\usepackage{amsthm}%
\usepackage{mathrsfs}%
\usepackage[title]{appendix}%
\usepackage{xcolor}%
\usepackage{textcomp}%
\usepackage{manyfoot}%
\usepackage{booktabs}%
\usepackage{algorithm}%
\usepackage{algorithmicx}%
\usepackage{algpseudocode}%
\usepackage{listings}%

\def\N{\mathbb N}
\def\A{\mathcal A}
\def\B{\mathcal B}

\def\LL{\mathcal L}
\def\uu{\mathbf{u}}
\def\aa{\mathbf{a}}
\def\bb{\mathbf{b}}
\def\vv{\mathbf{v}}
\def\colour{\mathrm{colour}}

\theoremstyle{definition}
\newtheorem{definition}{Definition}
\newtheorem{corollary}[definition]{Corollary}
\newtheorem{remark}[definition]{Remark}
\newtheorem{example}[definition]{Example}

\theoremstyle{plain}
\newtheorem{theorem}[definition]{Theorem}
\newtheorem{proposition}[definition]{Proposition}
\newtheorem{lemma}[definition]{Lemma}

\begin{document}

\title[2-balanced sequences coding rectangle exchange transformation]{2-balanced sequences coding rectangle exchange transformation}


\author*[1]{\fnm{Lubom\'ira} \sur{Dvo\v r\'akov\'a}}\email{lubomira.dvorakova@fjfi.cvut.cz}
\equalcont{These authors contributed equally to this work.}

\author[1]{\fnm{Zuzana} \sur{Mas\'akov\'a}}\email{zuzana.masakova@fjfi.cvut.cz}
\equalcont{These authors contributed equally to this work.}

\author[1]{\fnm{Edita} \sur{Pelantov\'a}}\email{edita.pelantova@fjfi.cvut.cz}
\equalcont{These authors contributed equally to this work.}

\affil*[1]{\orgdiv{Department of Mathematics}, \orgname{Faculty of Nuclear Sciences and Physical Engineering, Czech Technical University in Prague}, \country{Czech Republic}}


\abstract{We define a new class of ternary sequences that are $2$-balanced.  These sequences are obtained  by colouring of Sturmian sequences.  We show that the class contains sequences of any given letter frequencies. We provide an upper bound on factor and abelian complexity of these sequences.  Using the interpretation by rectangle exchange transformation, we prove that for almost all triples of letter frequencies, the upper bound on factor and abelian complexity is reached. The bound on factor complexity is given using a number-theoretical function which we compute explicitly for a~class of parameters.
}

\keywords{Sturmian sequence, Frequency, Balancedness, Numeration system, Factor complexity, Abelian complexity}


\pacs[MSC Classification]{68R15}

\maketitle

\section{Introduction}

In combinatorics on words, the most studied sequences are {\em Sturmian sequences}. These are binary $1$-{\em balanced} sequences, introduced by Hedlund and Morse~\citep{MoHe40}.
A sequence is $C$-{\em balanced} for some integer $C\geq 1$ if, for any two factors and each letter, the number of occurrences of that letter in the two factors differs at most by $C$.
There is a handy description of $1$-balanced sequences over larger alphabets by Hubert~\citep{Hubert00}: each recurrent aperiodic 1-balanced sequence can be obtained from a~Sturmian sequence by the so called colouring by constant gap sequences. For $2$-balanced sequences, no such useful characterization is known. 
Sturmian sequences are also characterized by their factor complexity, which is equal to $n+1$ for each $n \in \mathbb N$. The class of Sturmian sequences contains sequences of any given letter frequencies.

Sturmian sequences may be generalized to larger alphabets in many ways. Let us observe the properties of such generalizations focusing on factor complexity, balancedness and letter frequencies. The best known generalization of Sturmian sequences are {\em Arnoux-Rauzy (AR) sequences}~\citep{DrJuPi2001}.
They share a lot of properties with Sturmian sequences. Their factor complexity is equal to $(d-1)n+1$ for each $n \in \mathbb N$, where $d$ is the alphabet size. However, there exist AR sequences that are not $C$-balanced for any $C$. Berth\'e, Cassaigne, and Steiner~\citep{BeCaSt2013} found a sufficient condition for $2$-balancedness of AR sequences. 
It is known that the letter frequencies of any Arnoux-Rauzy sequence belong to the Rauzy gasket~\citep{ArSt13}, a fractal set of Lebesgue measure zero.
Another well-known generalization is represented by hypercubic billiard sequences generated by momenta with rationally independent components. Ternary hypercubic billiard sequences are 2-balanced~\citep{AnVi2022} and under some additional condition on momentum, their factor complexity equals $n^2+n+1$ (see~\citep{AR1994, ARTokyo1994, Be2007}). The class contains sequences of any given rationally independent frequencies. 
Sturmian sequences may be also defined as codings of exchange of two intervals, so called 2iet sequences. Codings of transformations exchanging more than two intervals therefore naturally generalize Sturmian sequences, too.
It is known that for example 3iet sequences have almost always factor complexity $2n+1$ for all $n \in \mathbb N$ and in such a case, they are not $C$-balanced for any $C$~\citep{AMP2021}.  

A new class of ternary sequences associated with a multidimensional continued fraction
algorithm was introduced and studied by Cassaigne, Labb\'e, and Leroy~\citep{CaLaLe2022}. Almost every sequence in the new class is $C$-balanced for some $C$. Moreover, the class contains sequences of any given frequencies.
A generalization of Sturmian sequences associated with $N$-continued fraction algorithms was introduced by Langeveld, Rossi, and Thuswaldner~\citep{LaRoTh2022}. All these sequences are binary and $C$-balanced for some $C$. It is possible to characterize those ones among them that are $2$-balanced. The class contains sequences of any given rationally independent frequencies. 
All sequences mentioned above, with exception of hypercubic billiard sequences, have sublinear factor complexity.     

In this paper, which is an extended version of the conference paper~\citep{DMP2023}, we define a new class of ternary 2-balanced sequences. We obtain these sequences by colouring of Sturmian sequences by a Sturmian sequence and a constant sequence. Already in the conference paper, we showed that our class contains 2-balanced sequences of any given letter frequencies (recalled in Section~\ref{sec:balancedness}). 
We provided there also an upper bound on factor and abelian complexity of these sequences (recalled in Section~\ref{sec:complexity}).  The factor complexity is at most quadratic. In this paper, we extend our study by the following results. In Section~\ref{sec:rectangle_coding} we prove that these sequences may be also seen as codings of rectangle exchange transformation. Using rectangle exchange transformation, we show that for almost all triples of letter frequencies, the upper bound on factor and abelian  complexity is reached. 
%The formula for factor complexity depends on the number of factors sharing the same Parikh vector in the coloured Sturmian sequence. In Section~\ref{sec:numeration_systems}, for the class of Sturmian sequences with slope whose reciprocal is a quadratic non-simple Parry unit, we deduce a formula based on expansion of integers in a particular non-standard numeration system. 
The formula for factor complexity depends on the number  of factors sharing the same Parikh vector in the coloured Sturmian sequence. We describe this quantity by a number-theoretic function $P_\alpha$, see Definition~\ref{d:Palpha} and Lemma~\ref{funkceP}.
In Section~\ref{sec:numeration_systems}, for the class of Sturmian sequences with slope whose reciprocal is a quadratic non-simple Parry unit, we deduce a formula for $P_\alpha$ based on expansion of integers in a particular non-standard numeration system.

\section{Preliminaries}
An {\em alphabet} $\mathcal A$ is a finite set of symbols, called {\em letters}. A {\em word} $w$ over $\mathcal A$ is a finite string of letters from $\mathcal A$. Its {\em length} $|w|$ equals the number of letters it contains. To denote the number of occurrences of a letter $a$ in $w$, we use $|w|_a$. The {\em Parikh vector} of $w$, denoted $\Psi(w)$, is a vector with coordinates defined by $\Psi(w)_a=|w|_a$. 
The set of all finite words over $\A$ together with the operation of concatenation forms a monoid, denoted $\A^*$.
Its neutral element is the \textit{empty word} $\varepsilon$. 
A {\em sequence} $\uu=u_0 u_1 u_2\cdots$ over $\A$ is an infinite string of letters from $\A$, i.e., $u_i \in \A$ for each $i \in \mathbb N$. A word $w$ is a {\em factor} of $\uu=u_0 u_1 u_2\cdots$ if there exists $i \in \mathbb N$ such that $w=u_i u_{i+1}\cdots u_{i+|w|-1}$. The {\em language} $\mathcal{L}_\uu$ of $\uu$ is the set of all factors of $\uu$.
A sequence $\uu$ is called {\em eventually periodic} if it can be written in the form $\uu=xy^{\omega}$, where $x,y$ are words, $y$ is non-empty, and $y^{\omega}$ denotes an infinite repetition of $y$. The sequence $\uu$ is {\em aperiodic} if $\uu$ is not eventually periodic. 
The {\em factor complexity} is a mapping $\mathcal{C}_\uu: \mathbb N \to \mathbb N$ defined for each $n \in \mathbb N$ as
$$\mathcal{C}_\uu(n)=\#\{w \ : \ |w|=n, \ w \in \mathcal{L}_\uu\}\,.$$
For each aperiodic sequence $\uu$, the factor complexity satisfies $\mathcal{C}_\uu(n)\geq n+1$ for all $n \in \mathbb N$~\citep{MoHe40}.
The {\em abelian complexity} is a mapping $\mathcal{C}^{ab}_\uu: \mathbb N \to \mathbb N$ defined for each $n \in \mathbb N$ as
$$\mathcal{C}^{ab}_\uu(n)=\#\{\Psi(w) \ : \ |w|=n, \ w \in \mathcal{L}_\uu\}\,.$$
Fici and Puzynina have recently published a comprehensive survey on abelian combinatorics on words~\citep{FiPu23}.

A \textit{morphism} over $\A$ is a mapping $\psi: \A^* \to \A^*$ such that $\psi(uv) = \psi(u) \psi(v)$ for all $u, v \in \A^*$.
The morphism $\psi$ can be naturally extended to sequences by setting
$\psi(u_0 u_1 u_2 \cdots) = \psi(u_0) \psi(u_1) \psi(u_2) \cdots\,$.

A sequence $\uu$ over $\A$ is $C$-{\em balanced} for some integer $C\geq 1$ if, for any two factors $u,v$ of $\uu$ of the same length and each letter $a\in \A$, we have $\bigl||u|_a-|v|_a\bigr|\leq C$. Binary $1$-{\em balanced} (also called {\em balanced}) sequences were introduced by Hedlund and Morse~\citep{MoHe40}. Binary balanced aperiodic sequences are called {\em Sturmian}.
Sturmian sequences may be equivalently defined as aperiodic sequences having the least possible factor complexity, i.e., $\mathcal{C}_\uu(n) = n + 1$ for all $n \in \mathbb N$.
The {\em frequency} $f(a)$ of a letter  $a\in \A$ in a sequence $\uu = u_0u_1u_2\cdots$ is defined as the limit, if it exists,
$$f(a)=\lim_{n\to \infty}\frac{|u_0 u_1 \cdots u_{n-1}|_a}{n}\,.$$
 If $\uu$ is a $C$-balanced sequence, then the frequency  $f(a)$ of each letter $a$ occurring in  $\uu$ is well defined and moreover each factor $u$ of $\uu$ satisfies  $\Bigl||u|_a - f(a)|u|\Bigr|\leq C$~\citep{Ad03, BeDe14}.

The abelian complexity of $\uu$ is bounded if and only if $\uu$ is $C$-balanced for some positive integer $C$~\citep{RiSaZa2010}.

Sturmian sequences may be equivalently defined as codings of two interval exchange~\citep{Ra79}.
For a~given  irrational parameter $\alpha \in (0,1)$, we consider two intervals  $I_a=[0,1-\alpha)$ and $I_b=[1-\alpha, 1)$. 
The \emph{two interval exchange transformation} (2iet) $T:[0,1) \to [0,1)$ is defined by
\begin{equation}\label{transformation}
T(x) = x +\alpha \mod 1 \ \ =  \ \ 
\left\{
\begin{array}{ll}
x + \alpha & \quad \text{if} \ x \in I_a,\\
x + \alpha -1 & \quad \text{if}  \ x \in I_b.
\end{array}
\right.
\end{equation}
If we take an initial point $ \rho \in I_a\cup I_b$, the sequence $\uu = u_0u_1u_2 \dots \in \{a,b\}^\N$ defined by
\begin{equation}\label{eq:slovoobecnyintercept}
u_n =
\left\{
\begin{array}{ll}
 a & \quad \text{if} \ T^n(\rho) \in  I_a,\\
 b & \quad \text{if} \ T^n(\rho)  \in I_b,
\end{array}
\right.  
\end{equation}
i.e., a coding of the trajectory of the point $\rho$, is a \textit{2iet sequence} with the \textit{parameter} $\alpha$. The sequence $\bigl(T^n(\rho)\bigr)_{n\in \N}$  is  uniformly distributed, in particular, it is dense in $[0,1]$. It is well known that the language of $\uu$ depends only on $\alpha$, but does not depend on $\rho$, and the frequency of letter $b$ in  $\uu$ is $\alpha$. 
\footnote{The parameter $\alpha$ is usually called {\em slope} of the Sturmian sequence~\citep{Lo2002}.}

If  the parameter $\alpha \in (0,1)$ in \eqref{transformation} is a rational number, then  $\bigl(T^n(\rho)\bigr)_{n\in \N}$  is a periodic sequence, and consequently, $\uu = u_0u_1u_2 \dots \in \{a,b\}^\N$ is a periodic $1$-balanced sequence. 

\begin{remark}\label{parikh}
The abelian complexity of a Sturmian sequence $\uu$ satisfies $\mathcal{C}^{ab}_\uu(n) = 2$ for all $n \geq 1$~\citep{CoHe73}. In other words, only two vectors occur as Parikh vectors of  factors of length $n$ in $\uu$.
Since  the frequency of letter $b$ in a Sturmian sequence $\uu$  coding the transformation $T$ is an irrational number, say $\alpha$, the average number of occurrences of letter $b$ in a factor of length $n$ is $\alpha\, n$. This fact implies that the number of occurrences of letter $b$ in a factor of length $n$ must be smaller than $\alpha n$ for some factors and larger than $\alpha n$ for other factors.  Balancedness of $\uu$ then implies that the Parikh vector of a factor $u$ of lenght $n$ is 
$$ \Psi(u) = \left( \begin{smallmatrix} \lfloor (1-\alpha)n\rfloor \\ \lceil n\alpha\rceil \end{smallmatrix} \right) \quad \text{or}\quad \Psi(u) = \left( \begin{smallmatrix} \lceil (1-\alpha)n\rceil \\ \lfloor n\alpha\rfloor   \end{smallmatrix} \right). 
$$
\end{remark}

\section{Colouring of sequences and balancedness}
\label{sec:balancedness}

\begin{definition} Let $\uu$ be a~sequence over the alphabet $\{a,b\}$.
Let $\bf a, \bf b$ be two sequences over mutually disjoint alphabets $\mathcal{A}$ and $\mathcal{B}$.
The colouring of $\uu$ by $\bf a, \bf b$ is a sequence $\vv = \colour(\uu,\aa, \bb)$ over $\A\cup \B$ obtained from $\uu$ by replacing the subsequence of all $a$’s with $\aa$ and all $b$'s with $\bb$.
\end{definition}

For ${\bf v}  = \colour(\uu, {\bf a},{\bf b})$ we use the notation $\pi(\vv) = \uu$ and $\pi(v) = u$ for any $v \in \LL(\vv)$ and the corresponding $u \in \LL(\uu)$.
We say that $\uu$ (resp. $u$) is a \emph{projection} of $\vv$ (resp. $v$).
The map $\pi : \LL(\vv) \to \LL(\uu)$ is clearly a morphism.
Indeed $\pi(\varepsilon) = \varepsilon$ and for every $v, v' \in \LL(\vv)$ one has $\pi(v v') = \pi(v) \pi(v')$.
In the sequel, we will use also the morphism $\Pi:(\A\cup \B)^*\to {\B}^*$ such that $\Pi(x)=\varepsilon$ if $x\in \A$ and $\Pi(x)=x$ if $x\in \B$. Clearly, for each $v \in \LL(\vv)$, it holds $|v|_x=|\Pi(v)|_x$ for each $x\in \mathcal B$ and $|\Pi(v)|=|\pi(v)|_{b}$.
\begin{example}\label{ex:Fib} Let $\uu$ be the Fibonacci sequence over $\{ a,b\}$, i.e., $\uu=\varphi(\uu)$ for the morphism $\varphi$ given by $\varphi(a)=ab, \ \varphi(b)=a$. Furthermore, let ${\bf a}=1^{\omega}$ and ${\bf b}=(23)^{\omega}$, then
$$
\begin{array}{rcl}
\uu &=& a\ b\ a\ a\ b\ a\ b\ a\ a\ b\ a\ a\ b\ a\ b\ a\ a\ b\ a\ b\ a\ a\ b\ a\ a\ b\ a\ b\ a\ a\ b\ a\ a\ b\ a\ b\ \cdots \\

{\bf v}&=& 1\  2\ 1\ 1\ 3\ 1\ 2\ 1\ 1\ 3\ 1\ 1\ 2\ 1\ 3\ 1\ 1\ 2\ 1\ 3\ 1\ 1\ 2\ 1\ 1\ 3\ 1\ 2\ 1\ 1\ 3\ 1\ 1\ 2\ 1\ 3\ \cdots
\end{array}
$$
We have  ${\Pi}({1211312113112})=23232$ and ${\pi}({1211312113112})= abaababaabaab$.
\end{example}
In the sequel, similarly as in Example~\ref{ex:Fib}, we intend to colour 2iet sequences by one constant sequence and one 2iet sequence. Let us observe what happens with balancedness after colouring.

\begin{lemma}\label{lem:2balance} Let  $\uu$ be a 1-balanced sequence over $\{ a,b\}$, and 
${\bf a}=1^{\omega}$ and ${\bf b} = b_0b_1b_2\dots$  be a 1-balanced sequence over $\{2,3\}$. Then the ternary sequence ${\bf v} = \colour(\uu, {\bf a},{\bf b}) $ is $2$-balanced.  
\end{lemma}

\begin{proof} 
Let $U,V$ be factors of $\bf v$ of the same length. We want to prove that for each letter $c \in  \{1,2,3\}$
$$\bigl||U|_c-|V|_c\bigr| \leq 2\,.$$
Denote $u=\pi(U)$ and $v=\pi(V)$. Clearly $|u|=|v|$. 
If $c=1$, then by definition of colouring, $|U|_1=|u|_a$ and $|V|_1=|v|_a$. By balancedness of $\uu$, we obtain $\bigl||U|_1-|V|_1\bigr|=\bigl||u|_a-|v|_a\bigr|  \leq 1$.
If $c \in \{2,3\}$, consider the morphism $\Pi:\{1,2,3\}^*\to {\{2,3\}}^*$ such that $\Pi(x)=\varepsilon$ if $x=1$ and $\Pi(x)=x$ if $x\in\{2,3\}$.
It holds that for each factor $W$ of $\bf v$ the word $\Pi(W)$ is a factor of $\bf b$.
By definition of $\Pi$, we have $|U|_c=|\Pi(U)|_c, \ |V|_c=|\Pi(V)|_c$ and by definition of colouring $|\Pi(U)|=|u|_{b}, \ |\Pi(V)|=|v|_{b}$.
By balancedness of $\uu$, the numbers $|u|_{b}$ and $|v|_{b}$ differ at most by one, hence the words $\Pi(U)$ and $\Pi(V)$ are factors of $\bf b$ whose lengths differ at most by one. 
\begin{itemize}
    \item Either $|\Pi(U)|=|\Pi(V)|$. Then since $\bf b$ is balanced, it follows $\bigl||\Pi(U)|_c-|\Pi(V)|_c\bigl|\leq 1$.
    \item Or (without loss of generality) $|\Pi(U)|=|\Pi(V)|+1$. Then $\Pi(U)=b_i\dots b_{i+\ell+1}$ and $\Pi(V)=b_j\dots b_{j+\ell}$ for some $i,j,\ell \in \mathbb N$. Then
    $$\bigl||\Pi(U)|_c-|\Pi(V)|_c\bigr|\leq \bigl||b_i\dots b_{i+\ell}|_c-|b_j\dots b_{j+\ell}|_c\bigr|+|b_{i+\ell+1}|_c\leq 2\,.$$
    
    \end{itemize}
   Since  $|U|_c=|\Pi(U)|_c$ and $|V|_c=|\Pi(V)|_c$, we have proved $\bigl||U|_c-|V|_c\bigr| \leq 2$.
\end{proof}

Let us show that by colouring a 2iet sequence by one constant sequence and one 2iet sequence, we are able to construct $2$-balanced sequences with prescribed frequencies.

\begin{theorem}\label{theorem} Let $f(1), f(2), f(3)$ be positive numbers such that $f(1)+f(2)+f(3)=1$. Then there exist binary $1$-balanced sequences $\uu, \bb$ and a constant sequence ${\bf a}$ such that the ternary sequence $\vv = \colour(\uu, {\bf a}, {\bf b})$  satisfies  
\begin{enumerate}
    \item the frequency of letter $i$ in ${\bf v}$ is $f(i)$ for each $i\in \{1,2,3\}$;
    \item ${\bf v}$ is $2$-balanced.
\end{enumerate}
\end{theorem}
\begin{proof} 
\noindent We will make use of a simple fact that letter frequencies in $\vv = \colour(\uu, {\bf a}, {\bf b})$ can be easily computed from letter frequencies in $\uu$ and ${\bf b}$. 

We set $\vv=\colour(\uu, \aa, \bb)$, where $\aa=1^{\omega}$, $\uu$ is a 2iet sequence over $\{a,b\}$ with frequency of $a$ being $f(1)$ and frequency of $b$ being $\alpha=1-f(1)$ and $\bb$ is a 2iet  sequence over $\{2,3\}$, where the frequency of $2$ equals $\frac{f(2)}{\alpha}$ and the frequency of $3$ equals $\frac{f(3)}{\alpha}$.

\medskip 
$2$-balancedness of $\vv$ is a consequence of Lemma~\ref{lem:2balance}.
\end{proof}

\section{On factor and abelian complexity}\label{sec:complexity}

Using Theorem \ref{theorem} we can construct for every frequency vector $\vec{f} \in \mathbb{R}^{3}$ a 2-balanced  sequence  $\vv = \colour(\uu, {\bf a}, {\bf b})$ having the required frequencies of letters. If the frequency vector is rational, then our construction gives a periodic infinite sequence, in particular,  a sequence with bounded factor complexity.  

In this section we  find an upper bound on factor complexity and on abelian complexity   of $\vv $ in the case when the binary sequences we used in the construction are  Sturmian sequences.   For this purpose, we need to determine how  many factors of length $n$ in a Sturmian sequence have the same Parikh vector. The form of the Parikh vectors is described in Remark~\ref{parikh}.

\begin{definition} \label{d:Palpha}
Let $\alpha \in (0,1)$, $\alpha$ irrational, and let $\uu$ be a Sturmian sequence over $\{a,b\}$ with frequency of letter $b$ equal to  $\alpha$.  For each $n \in \N, n\geq 1$,  we denote   
$$
P_\alpha(n) = \# \left\{ u \in \mathcal{L}(\uu):  \Psi(u) = \left( \begin{smallmatrix} \lfloor (1-\alpha)n\rfloor \\ \lceil n\alpha\rceil \end{smallmatrix} \right)\right\}. 
$$
\end{definition}

\noindent Since the factor complexity of a Sturmian sequence is  $\mathcal{C}_\uu(n)=n+1$ for every $n \in \N$, obviously  
$$
n+1 - P_\alpha(n) =\# \left\{ u \in \mathcal{L}(\uu): \Psi(u) = \left( \begin{smallmatrix} \lceil (1-\alpha)n\rceil \\ \lfloor n\alpha\rfloor   \end{smallmatrix} \right)\right\}. 
$$
The following formula for computation of $P_\alpha$ was derived already by Rigo et al.~\citep{RiSaVa2013} and Fici et al.~\citep{Fici2016}.
For reader's convenience, we provide here a simple proof based on 2 interval exchange transformation. 
\begin{lemma}\label{funkceP} 
Let $\alpha \in (0,1)$, $\alpha$ irrational.   Then for each  $n \in \N, n\geq 1$, 
$$
P_\alpha(n) = \# \Bigl\{k\in \{1,\ldots,n\}: \ k\alpha - \lfloor k\alpha \rfloor\  \leq \  n\alpha - \lfloor n\alpha \rfloor  \Bigr\}\,.  
$$
\end{lemma}

\begin{proof} Let $T$ denote the transformation defined in \eqref{transformation} and let $\uu$ be a coding of $T$.  
A~word $x_0x_1 \cdots x_{n-1}$ is a factor of $\uu$ if and only if there exists $x \in [0,1)$ such that

\medskip
\centerline{$x\in I_{x_0}, \  T(x) \in I_{x_1}, \ \ldots, \  T^{n-1}(x) \in I_{x_{n-1}}$.}

\medskip

Consider a partition of $[0,1)$ into $n+1$ intervals closed from the left and open from the right, the partition is  given by the $n$ points $T^{-k}(1-
\alpha)$,  for $k= 0,1,\ldots, n-1$. Let  $x, y \in [0,1)$. 
Then  obviously, the symbolic coding of the $n$-tuples  \begin{equation}\label{nTuples}x,T(x), T^2(x),  \ldots, T^{n-1}(x)\quad \text{and}\quad y,T(y), T^2(y),  \ldots, T^{n-1}(y)\end{equation} coincide if and only if $x$ and $y$ belong to the same interval of the partition.

Let $[e,f)$  and $[f,g)$ be two neighbouring intervals of the partition, where $f = T^{-j}(1-\alpha)$ for some $j\leq n-1$. Let the words $x_0x_1\cdots x_{n-1}$  and $y_0y_1\cdots y_{n-1}$ code the $n$-tuples listed in \eqref{nTuples} for  $x \in [e,f)$  and $y\in [f,g)$.

If $j < n-1$, then $x_i = y_i$ for every $i\neq j$ and $i \neq j+1$. Moreover, $x_jx_{j+1} = ab$  and  $y_jy_{j+1} = ba$. 
Hence the Parikh vectors of $x_0x_1\cdots x_{n-1}$  and $y_0y_1\cdots y_{n-1}$ coincide. 

If $ j= n-1$, then $x_{n-1} = a$,  $y_{n-1} = b$, and  $x_i = y_i$ for every $i< n-1$. Hence the number of letters $a$ in  $x_0x_1\cdots x_{n-1}$ is by one greater than the number of letters $a$ in $y_0y_1\cdots y_{n-1}$. 

Thus the factors of length $n$ have only two distinct Parikh vectors. Those intervals of the partition that are situated to the left from the point $ T^{-n+1}(1-\alpha)$ correspond to factors with a smaller number of $b$'s, those ones situated to the right from $ T^{-n+1}(1-\alpha)$ correspond to factors  with a greater number of $b$'s.
Hence
$$P_\alpha(n) = \bigl\{k\in \{1, \ldots, n\}: T^{-k+1}(1-\alpha) \geq  T^{-n+1}(1-\alpha)\bigr\}\,.  $$

The inverse transformation to $T$ is $T^{-1}(x) = x-\alpha \mod 1$. Hence $ T^{-k+1}(1-\alpha) = 1-k\alpha  \mod 1 = 1-k\alpha  - \lfloor 1- k\alpha\rfloor = 1- k\alpha +\lfloor k\alpha \rfloor $. Consequently, $T^{-k+1}(1-\alpha) \geq  T^{-n+1}(1-\alpha)$ if and only if $ k\alpha -\lfloor k\alpha \rfloor \leq n\alpha -\lfloor n\alpha \rfloor$. 
\end{proof}

\begin{example}\label{Pocet} 
Using Lemma \ref{funkceP} we computed  the values of $P_\alpha(n)$  for  $\alpha= 2-\sqrt{3}$ 
and  $n \in \N, n\leq 10$. They are provided in Table \ref{Table0}. 

\begin{table}[ht]
    \centering
        \setlength{\tabcolsep}{3pt}
\renewcommand{\arraystretch}{1.2}

\begin{tabular}{|c|c|c|c|c|c|c|c|c|c|c|}
    \hline
     $n$ & 1 & 2 & 3 & 4 & 5 & 6 & 7 & 8 & 9 & 10 \\ 
    \hline
    $n\alpha-\lfloor n\alpha \rfloor$ & 0.2679 & 0.5359 & 0.8038 & 0.0718 & 0.3397 & 0.6077 & 0.8756 & 0.1436 & 0.4115 & 0.6795 \\ 
    \hline
    $P_\alpha(n)$ & 1 & 2 & 3 & 1 & 3 & 5 & 7 & 2 & 5 & 8 \\ 
    \hline
\end{tabular}
\caption{$\alpha=2-\sqrt{3}  \sim 0,2679 $}\label{Table0}
\end{table}
\end{example}

\subsection{Ternary 2-balanced sequences}
\begin{theorem}\label{ternary}  
Let $\uu$ be a  Sturmian sequence over $\{a,b\}$ with frequency of letter $b$ equal to an irrational number  $\alpha \in (0,1)$. Let  $\aa=1^\omega$ and $\bb$  be a Sturmian sequence over the alphabet $\{2,3\}$. Then the sequence ${\bf v}  = \colour(\uu, {\bf a},{\bf b})$ has the following properties:
\begin{itemize}
\item $\vv$ is $2$-balanced; moreover, $\vv$ is not 1-balanced if $\alpha\not \in \{\frac{1}{1+\gamma}, \frac{1}{2-\gamma}\}$, where $\gamma \in (0,1)$ is the frequency of letter $3$ in the sequence $\bb$; 
\item the factor and abelian complexity satisfy  
\begin{equation}\label{upperbound}\mathcal{C}_\vv (n) \leq P_\alpha(n) +  (n+1) \lceil n\alpha \rceil\qquad \text{and} \qquad \mathcal{C}^{ab}_\vv (n) \leq 4, 
\end{equation}
 for each $n \in \N, n\geq 1$.  
 \end{itemize}
\end{theorem}
\begin{proof}
The sequence $\vv$ is $2$-balanced by Lemma~\ref{lem:2balance}. 
On one hand, using a theorem by Hubert~\citep{Hubert00}, each ternary $1$-balanced aperiodic recurrent sequence over $\{\hat 1, \hat 2, \hat 3\}$ is a colouring of a Sturmian sequence by sequences $(\hat 1\hat 2)^{\omega}$ and ${\hat 3}^{\omega}$. The frequency vector of the colouring has then evidently two equal coordinates. On the other hand, the frequency vector of the sequence $\vv$ is equal to $(1-\alpha,\alpha(1-\gamma), \alpha\gamma)$. The last two coordinates are not equal by irrationality of $\gamma$. The reader easily verifies that if there are two equal coordinates, then $\alpha=\frac{1}{1+\gamma}$ or $\alpha=\frac{1}{2-\gamma}$. Consequently, $\vv$ is not $1$-balanced for $\alpha$ distinct from those two values.

Fix a factor $w\in \mathcal{L}(\uu)$ of length $n$.  We colour all letters $b$ in $w$ by a  factor of $\bb$ of length $|w|_b$. As  $\bb$ is Sturmian, we can use $|w|_b+1$ factors of $\bb$ for colouring $w$. Thus,  colouring of the fixed factor $w$ gives us at most $|w|_b  +1$ distinct factors of $\vv$. Since in $\uu$ we have $P_\alpha(n)$ factors with  $|w|_b =  \lceil n \alpha \rceil$  and $n+1-P_\alpha(n)$ factors with $|w|_b =  \lfloor n \alpha \rfloor$, the number of factors of the coloured sequence $\vv$ is at most 
$$\bigl(1+ \lceil n \alpha \rceil\bigr) P_\alpha(n) + \bigl(1+ \lfloor n \alpha \rfloor\bigr) \bigl(n+1-P_\alpha(n)\bigr) =  P_\alpha(n)\bigl( \lceil n \alpha \rceil -\lfloor n \alpha \rfloor \bigr) + (n+1)(1+ \lfloor n \alpha \rfloor)\,.
$$

There are two distinct Parikh vectors for factors of length $n$ in $\uu$. After colouring, each of them gives rise to at most two distinct Parikh vectors in $\vv$. 
Consequently, $\mathcal{C}^{ab}_\vv (n) \leq 4$.
\end{proof}

\begin{remark}\label{rem:asymptotic behavior}
Since $P_{\alpha}(n)\leq n $ for each $n \in \N$, it follows from Theorem~\ref{ternary} that $$\mathcal{C}_\vv (n) \leq  \alpha n^2\bigl(1+ o(1)\bigr)\,.$$
\end{remark}

\begin{remark} Notice that the fractions $\frac{1}{1+\gamma}$ and $\frac{1}{2-\gamma}$ in the first part of Theorem~\ref{ternary} are larger than $\frac{1}{2}$. Recall that  $\alpha$ is the frequency of letter $b$ in $\uu$ and $\gamma$ is the frequency of  letter $3$ in $\bb$. Consequently, as soon as $\alpha < \frac12$, the sequence $\vv$ is not 1-balanced. Let us consider the case $\alpha$ equals $\frac{1}{1+\gamma}$ or $\frac{1}{2-\gamma}$. If $\bb$ and $\uu$ are standard  Sturmian sequences,  then the relation between the continued fraction of $\alpha$ and $\gamma$ implies that $\uu$ over $\{a,b\}$ is a morphic image of  $\bb$ over $\{2,3\}$  under the morphism $2 \mapsto ba, \ 3 \mapsto b$, or under the morphism $2 \mapsto b, \ 3 \mapsto ba$. 
\end{remark}

\begin{example} \label{ex:ruzna_alpha_gamma}
We compare the factor and abelian complexity with  the upper bounds from Theorem~\ref{ternary} for ${\bf v}  = \colour(\uu, {\bf a},{\bf b})$, where $\uu$ is a Sturmian sequence over $\{a,b\}$ and $\alpha$ is the frequency of $b$, $\bb$ is a~Sturmian sequence over $\{2,3\}$ and $\gamma$ is the frequency of letter $3$ in $\bb$ and $\aa=1^{\omega}$, for three cases, see Tables~\ref{table1}, \ref{table2}, \ref{table3}. Recall that $\tau$ denotes the golden mean, i.e., $\tau=\frac{1+\sqrt{5}}{2}$.
We observe that in the first and second case, where $\alpha$ and $\gamma$ are from different quadratic fields, the upper bound on factor complexity from Theorem~\ref{ternary} is reached and $\mathcal{C}^{ab}_\vv (n) =4$ for each sufficiently large $n$. In the sequel, we will show that this is true if the numbers $1, \alpha, \alpha\gamma$  are rationally independent. In Table~\ref{table3}, where $\alpha=1-\gamma =\frac{1}{\tau}$, the upper bounds are not reached.
\begin{table}[ht]
    \centering
        \setlength{\tabcolsep}{3pt}
\renewcommand{\arraystretch}{1.2}
\setlength{\tabcolsep}{4pt}
\begin{tabular}{|c|c|c|c|c|c|c|c|c|c|c|c|c|c|c|c|c|}
    \hline
     $n$ & 1 & 2 & 3 & 4 & 5 & 6 & 7 & 8 & 9 & 10 & 11 & 12 & 13 & 14 & 15 & 16 \\ 
    \hline
    $P_\alpha(n)+(n+1)\lceil{n\alpha}\rceil$ & 3 & 7 & 11 & 17 & 25 & 33 & 43 & 53 & 65 & 79 & 93 & 109 & 127 & 145 & 165 & 185\\ 
    \hline
    $\textit{C}_\textbf{v}(n)$ & 3 & 7 & 11 & 17 & 25 & 33 & 43 & 53 & 65 & 79 & 93 & 109 & 127 & 145 & 165 & 185\\ 
    \hline
    $\mathcal{C}^{ab}_\vv (n)$  & 3 & 4 & 4 & 4 & 4 & 4 & 4 & 4 & 4 & 4 & 4 & 4 & 4 & 4 & 4 & 4\\ 
    \hline
\end{tabular} 
\caption{$\alpha=\frac{1}{\tau}$ and $\gamma=3-2\sqrt{2}$}\label{table1}
\end{table}

\begin{table}[ht]
    \centering
        \setlength{\tabcolsep}{3pt}
\renewcommand{\arraystretch}{1.2}
\setlength{\tabcolsep}{4pt}
\begin{tabular}{|c|c|c|c|c|c|c|c|c|c|c|c|c|c|c|c|c|}
    \hline
     $n$ & 1 & 2 & 3 & 4 & 5 & 6 & 7 & 8 & 9 & 10 & 11 & 12 & 13 & 14 & 15 & 16 \\ 
    \hline
    $P_\alpha(n)+(n+1)\lceil{n\alpha}\rceil$ & 3 & 5 & 7 & 9 & 11 & 15 & 19 & 23 & 27 & 31 & 35 & 41 & 47 & 53 & 59 & 65\\ 
    \hline
    $\mathcal{C}_\vv(n)$ & 3 & 5 & 7 & 9 & 11 & 15 & 19 & 23 & 27 & 31 & 35 & 41 & 47 & 53 & 59 & 65\\ 
    \hline
    $\mathcal{C}^{ab}_\vv (n)$  & 3 & 3 & 3 & 3 & 3 & 4 & 4 & 4 & 4 & 4 & 4 & 4 & 4 & 4 & 4 & 4\\ 
    \hline
\end{tabular} 
\caption{$\alpha=3-2\sqrt{2}$ and $\gamma=\frac{1}{\tau}$}\label{table2}
\end{table}

\begin{table}[ht]
    \centering
        \setlength{\tabcolsep}{3pt}
\renewcommand{\arraystretch}{1.2}
\setlength{\tabcolsep}{3.8pt}
\begin{tabular}{|c|c|c|c|c|c|c|c|c|c|c|c|c|c|c|c|c|}
    \hline
     $n$ & 1 & 2 & 3 & 4 & 5 & 6 & 7 & 8 & 9 & 10 & 11 & 12 & 13 & 14 & 15 & 16\\ 
    \hline
    $P_\alpha(n)+(n+1)\lceil{n\alpha}\rceil$ & 3 & 7 & 11 & 17 & 25& 33 & 43 & 53 & 65 & 79 & 93 & 109 & 127 & 145 & 165 & 185\\ 
    \hline
    $\mathcal{C}_\vv(n)$ & 3 & 4 & 5 & 6 & 7 & 8 & 9 & 10 & 11 & 12 & 13 & 14 & 15 & 16 & 17 & 18\\ 
    \hline
    $\mathcal{C}^{ab}_\vv (n)$  & 3 & 3 & 3 & 4 & 3 & 3 & 3 & 3 & 4 & 3 & 3& 4 & 3 & 3 & 3 & 3\\ 
    \hline
\end{tabular} 
\caption{$\alpha=\frac{1}{\tau}$ and $\gamma=\frac{1}{\tau^2}$}\label{table3}
\end{table}
\end{example}
\section{Ternary 2-balanced sequences as coding of rectangle exchange transformation}
\label{sec:rectangle_coding}

In this section we will consider, for simplicity of notation, Sturmian sequences with intercept 0, both for the sequence  $\uu=u_0 u_1 u_2 \cdots$ and for the sequence $\bb$
that is used for colouring of $\uu$.

Recall that every Sturmian sequence is a coding of a 2 interval exchange transformation, see~\eqref{transformation}. 
Using the fractional part $\{x\}=x-\lfloor x\rfloor = x \!\!\! \mod 1 $  of a real number $x$, we can write
\begin{equation}\label{eq:unfrac} 
u_n =    
\left\{
\begin{array}{ll}
a  & \quad \text{if} \  \{\alpha n \}  \in [0, 1-\alpha),\\
b  & \quad \text{if}  \ \{\alpha n \}  \in [ 1-\alpha, 1).
\end{array}
\right.
\end{equation}
Note that 
\begin{equation} 
u_n =    
\left\{
\begin{array}{ll}
a  & \quad \text{if and only if  \ \ } \  \lfloor(n+1)\alpha  \rfloor - \lfloor n\alpha  \rfloor = 0 ,\\
b  & \quad \text{if and only if  \ \ }  \ \lfloor(n+1)\alpha  \rfloor - \lfloor n\alpha  \rfloor = 1. 
\end{array}
\right.
\end{equation}
The number of occurrences of the letter $b$ in the prefix of $\uu$ of length $n$ is equal to 
$$
\sum_{k=0}^{n-1}\Bigl(\lfloor(k+1)\alpha  \rfloor - \lfloor k\alpha  \rfloor\Bigr)  =  \lfloor n\alpha  \rfloor. 
$$
Let us express the $n^{th}$ letter of the sequence $\vv = \colour(\uu, {\bf a},{\bf b})$, where $\aa = 1^\omega$
and $\bb$ is a Sturmian sequence over $\{2,3\}$ with $\gamma$ being the frequency of letter $3$. If $u_n = a$, then $v_n = 1$. Let now $u_n = b$. Since for colouring of $\lfloor n\alpha  \rfloor$  letters $b$ of the prefix $u_0u_1 \cdots u_{n-1}$ we have already used the prefix  $b_0b_1 \cdots b_{\lfloor n\alpha  \rfloor - 1}$, the $n^{th}$ letter of $\vv$ will obtain the colour $b_{\lfloor n\alpha  \rfloor}$, i.e., by~\eqref{eq:unfrac}
\begin{equation}\label{eq:tvarV}
v_n = \left\{
\begin{array}{ll}
1  & \quad \text{if} \  \  \{\alpha n \}  \in [0, 1-\alpha),\\
2  & \quad \text{if}  \ \ \{\alpha n \}  \in [ 1-\alpha, 1) \  \ \text{and} \  \  \{\gamma \lfloor \alpha n \rfloor \}  \in [0, 1-\gamma), \\
3 & \quad \text{if}  \ \ \{\alpha n \}  \in [ 1-\alpha, 1) \  \ \text{and} \  \  \{\gamma \lfloor \alpha n \rfloor \}  \in [1-\gamma, 1).
\end{array}
\right.
\end{equation}
Define the following  three rectangles in $\mathbb{R}^2$: 
$$
\begin{array}{l}
R_1= \{(x,y) \in\mathbb{R}^2 \ : \ x \in [0, 1-\alpha)\ \text{and}\  y \in [0,1) \},  \\
R_2 = \{(x,y) \in\mathbb{R}^2 \ : \ x \in [1-\alpha, 1)\ \text{and}\  y \in [0,1-\gamma) \},\\  
R_3 = \{(x,y) \in\mathbb{R}^2 \ : \ x \in [1-\alpha, 1)\ \text{and}\  y \in [1-\gamma, 1) \}.\\  
\end{array}
 \quad
$$

We will show that the sequence $\vv = \colour(\uu, {\bf a},{\bf b})$ is a coding of the rectangle exchange $S: [0,1)^2 \mapsto[0,1)^2 $, given by
\begin{equation}\label{eq:RectTrans}
 S(x,y) =     
\left\{
\begin{array}{ll}
(x+\alpha \!\!\!\mod 1, \ y)  & \quad \text{if} \  x  \in [0, 1-\alpha),\\
(x+\alpha \!\!\!\mod 1, \  y + \gamma \!\!\!\mod 1)  & \quad \text{if} \  x  \in [1-\alpha, 1).\\
\end{array}
\right.   
\end{equation}

The transformation $S$ is depicted in Figure~\ref{f:transformace}.

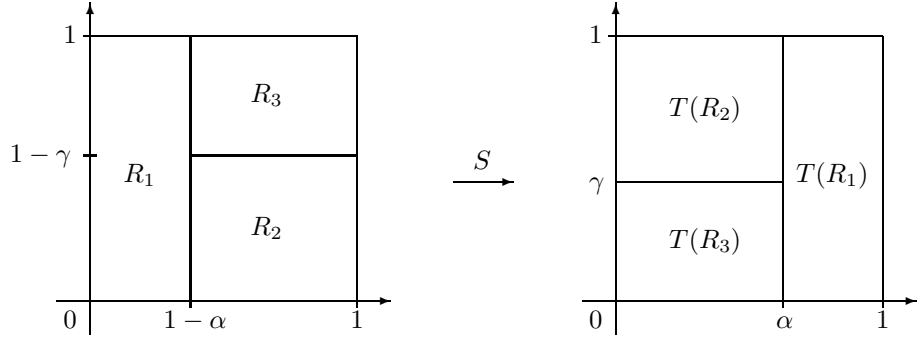
\begin{figure}
    \centering
{%\large
%\hfill
\begin{picture}(120,120)
\setlength{\unitlength}{2.5pt}
\put(0,5){\vector(1,0){50}}
\put(5,0){\vector(0,1){50}}
\put(45,4){\line(0,1){41}}
\put(4,45){\line(1,0){41}}
\put(20,4){\line(0,1){41}}
\put(20,27){\line(1,0){25}}
\put(4,27){\line(1,0){2}}
\put(10,23){$R_1$}
\put(29,15){$R_2$}
\put(29,35){$R_3$}
\put(1,1){$0$}
\put(44,1){$1$}
\put(1,44){$1$}
\put(16,1){$1-\alpha$}
\put(-7,26){$1-\gamma$}
\end{picture}
\qquad
\begin{picture}(30,120)
\setlength{\unitlength}{2.5pt}
\put(2,23){\vector(1,0){9}}
\put(5,25){$S$}
\end{picture}
\qquad
\begin{picture}(120,120)
\setlength{\unitlength}{2.5pt}
\put(0,5){\vector(1,0){50}}
\put(5,0){\vector(0,1){50}}
\put(45,4){\line(0,1){41}}
\put(4,45){\line(1,0){41}}
\put(30,4){\line(0,1){41}}
\put(5,23){\line(1,0){25}}
\put(32,23){$T(R_1)$}
\put(13,33){$T(R_2)$}
\put(13,13){$T(R_3)$}
\put(1,1){$0$}
\put(44,1){$1$}
\put(1,44){$1$}
\put(29,1){$\alpha$}
\put(1,22){$\gamma$}
\end{picture}
%\hfill
}
    \caption{Transformation S from~\eqref{eq:RectTrans}} 
\label{f:transformace}
\end{figure}

\begin{proposition} \label{p:rectangle}
Let $\uu$ be a Sturmian sequence over $\{a,b\}$ with intercept 0 and with frequency of letter $b$ equal to $\alpha$. Let 
$\vv = \colour(\uu, {\bf a},{\bf b})$, where $\aa = 1^\omega$
and $\bb$ is a~Sturmian sequence over $\{2,3\}$ with intercept 0 and with frequency of letter $3$ equal to $\gamma$.
Then for every $n \in \N$  and every  letter $i \in  \{1,2, 3\}$  we have  
$$ 
v_n = i \qquad \text{if and only if} \qquad S^n(0,0) \in R_i, 
$$
where $S$ is the transformation defined in~\eqref{eq:RectTrans}. 
\end{proposition}

\begin{proof}
Thanks to~\eqref{eq:tvarV}, it is sufficient to demonstrate that for any $n\in \N$ it holds that 
$S^n(0,0) = \Bigl( \{\alpha n\},  \ \{\gamma \lfloor \alpha n\rfloor\} \Bigr)$. We use induction on $n \in \N$. 

The step $n =0$ is trivial. Consider $S^{n+1}(0,0)$. 
% Zapis 
% $$ 
% \{z\} = z \!\!\! \mod 1 \ \  \ \ \text{implikuje}\ \  \ \ \{k z\} + z  \!\!\! \mod 1= \{(k+1)z \} \ \ \text{pro kazde } \ z \in \mathbb{R}  \ \text{and} \ k \in \mathbb{Z}
% $$
By induction hypothesis, we obtain
$S^{n+1}(0,0) = S\bigl(S^n(0,0)\bigr)  = S\Bigl( \{\alpha n\},  \ \{\gamma \lfloor \alpha n\rfloor\} \Bigr)$. Denote $(x,y)=\Bigl( \{\alpha n\},  \ \{\gamma \lfloor \alpha n\rfloor\} \Bigr)$.  Clearly $x+ \alpha \!\!\! \mod 1 = \{\alpha n\} +\alpha \!\!\!\mod 1 = \{\alpha(n+1)\}$. 

\begin{itemize}
    \item If $x=\{\alpha n\} \in [0, 1-\alpha) $, then  $\lfloor(n+1)\alpha  \rfloor = \lfloor n\alpha  \rfloor $. 
    
    Therefore $y =\{\gamma \lfloor \alpha n\rfloor\}= \{\gamma \lfloor \alpha (n+1)\rfloor\}$ and $S(x,y) = \Bigl( \{\alpha (n+1)\},  \ \{\gamma \lfloor \alpha (n+1)\rfloor\} \Bigr)$.

    \item If $x=\{\alpha n\} \in [1-\alpha, 1) $, then  $\lfloor(n+1)\alpha  \rfloor = 1+\lfloor n\alpha  \rfloor $. 

    Therefore $y+\gamma \!\!\! \mod 1 = \{\gamma \lfloor \alpha n\rfloor\} + \gamma \!\!\! \mod 1  = 
    \{\gamma (1+\lfloor \alpha n\rfloor)\} =    \{\gamma \lfloor \alpha (n+1)\rfloor\}$.
Again  $S(x,y) = \Bigl( \{\alpha (n+1)\},  \ \{\gamma \lfloor \alpha (n+1)\rfloor\} \Bigr)$, as desired. 
\end{itemize}
\end{proof}

\begin{remark}
Using similar, but more technical argumentation, we can derive that the colouring 
$\vv = \colour(\uu, {\bf a},{\bf b})$ 
of a Sturmian sequence $\uu$ with intercept $\rho_1$
by $\aa = 1^\omega$
and $\bb$,  which is a Sturmian sequence over $\{2,3\}$ with intercept $\rho_2$, satisfies
\begin{equation}
v_n = \left\{
\begin{array}{ll}
1  & \quad \text{if} \  \  \{\alpha n +\rho_1\}  \in [0, 1-\alpha),\\
2  & \quad \text{if}  \ \ \{\alpha n +\rho_1\}  \in [ 1-\alpha, 1) \  \ \text{and} \  \  \{\gamma \lfloor \alpha n \rfloor +\rho_2\}  \in [0, 1-\gamma), \\
3 & \quad \text{if}  \ \ \{\alpha n +\rho_1\}  \in [ 1-\alpha, 1) \  \ \text{and} \  \  \{\gamma \lfloor \alpha n \rfloor +\rho_2\}  \in [1-\gamma, 1),
\end{array}
\right.
\end{equation}
and $\vv$ is a coding of the point $(\rho_1,\rho_2)$ under the same rectangle exchange $S$.
\end{remark}

\begin{proposition}\label{p:huste}
If $S$ is the transformation defined in~\eqref{eq:RectTrans} and the three numbers  $1, \alpha$ and  $ \alpha\gamma$ are rationally independent, then for every $(x_0,y_0) \in [0,1)^2$ the set  $\{S^n(x_0,y_0) : n \in \N\}$ is dense in  $[0,1)^2$.       
\end{proposition}

\begin{proof}
In~\citep{Haller}, it is shown that the rectangle exchange transformation defined in~\eqref{eq:RectTrans} is measure-theoretically isomorphic to the rotation $R_\theta$ on the torus $\Omega={\mathbb R}^2/{\mathbb Z}^2=[0,1)\times[0,1)$ by the vector $\theta=(\alpha,\alpha\gamma)$. In particular, one can check that $\Phi\circ R_{\theta} = S\circ\Phi$, where $\Phi:\Omega\to\Omega$ is given by 
$$
\Phi(x,y) = (x, y -\gamma x \!\!\!\mod 1) .
$$
From that, we derive that for every $k\geq 1$, $\Phi\circ R_\theta^k = S^k\circ\Phi$.
Consequently, the set $\{S^k(x_0,y_0):k\in\N\}$ is dense in $\Omega$, whenever $\{ R_\theta^k(x_1,y_1) : k\in\N \}$ is dense in  $\Omega$, where $(x_1,y_1)=\Phi^{-1}(x_0,y_0)=(x_0,y_0+\gamma x_0 \!\!\!\mod 1)$.

Now we can use the famous Kronecker Approximation Theorem in ${\mathbb R}^2$, see e.g.~\citep[p. 15]{DrTi97}, which states
that
if $\theta=(\theta_1,\theta_2)\in{\mathbb R}^2$ is such that $1,\theta_1,\theta_2$ are rationally independent,  then for any $(x_1,y_1)\in\Omega$ the sequence $\big(R_\theta^k(x_1,y_1)\big)_{k\geq 0}$ is dense in $\Omega$. 
\end{proof}

Using the interpretation by rectangle exchange transformation, we are able to provide conditions on letter frequencies guaranteeing that the upper bound on factor and abelian complexity from Theorem~\ref{ternary} is reached. Example~\ref{ex:ruzna_alpha_gamma} may serve as illustration of this result.

\begin{theorem} \label{thm:ternary} 
Let $\uu$ be a Sturmian sequence over $\{a,b\}$ with frequency of letter $b$ equal to $\alpha$. Let 
$\vv = \colour(\uu, {\bf a},{\bf b})$, where $\aa = 1^\omega$ and $\bb$ is a Sturmian sequence over $\{2,3\}$ with frequency of letter $3$ equal to $\gamma$.
Suppose that 1, $\alpha$ and $\alpha\gamma$ are rationally independent. Then 
$$
\mathcal{C}_\vv (n) = P_\alpha(n) +  (n+1) \lceil n\alpha \rceil\qquad \text{ for each }n\in\N, \, n \geq 1, 
$$
and 
$$
\mathcal{C}^{ab}_\vv (n) = 4 \qquad \text{ for each } n \in \N, n\geq \big\lceil\tfrac1\alpha\big\rceil .  
$$
\end{theorem}

\begin{proof} 
%Recall that the language of Sturmian sequences does not depend on the intercept. For computation of complexities we can therefore consider that both $\uu$ and $\bb$ are Sturmian sequences with intercept 0.
%
Realize that projection of the transformation $S$ defined in~\eqref{eq:RectTrans} to the first coordinate is equal to the exchange $T$ of two intervals with parameter $\alpha$ as defined in~\eqref{transformation}. At the same time, projection of $S$ to the second coordinate is either constant, or equal to the exchange of two intervals with parameter $\gamma$. In order to distinguish between them, we will write $T_\alpha$, $T_\gamma$, respectively.
% More formally, define for each $x\in[0,1)$ the transformation 
% $$
% T_x=\begin{cases}
% \textit{id} & \text{ if } x\in[0,1-\alpha),\\
% R_\gamma    & \text{ if } x\in[1-\alpha,1).
% \end{cases}
% $$
% Then we can write 
% \begin{equation}\label{eq:Sprojekcenaslozky}
% S(x,y) = (R_\alpha(x),T_x(y)).
% \end{equation}

For an arbitrary fixed $w\in\LL(\uu)$ of length $|w|=n$ and arbitrary $z\in\LL(\bb)$ of length $|z|=|w|_b$ we wish to find a factor $v\in\LL(\vv)$ of length $n$ such that $\pi(v)=w$ and $\Pi(v)=z$, where $\pi:\{1,2,3\}^*\to\{a,b\}^*$ and $\Pi:\{1,2,3\}^*\to\{2,3\}^*$ are projections defined 
in Section~\ref{sec:balancedness}.

% Denote for the moment by $\uu_x$ the infinite sequence over $\{a,b\}$ coding the point $x\in[0,1)$ under the transformation $R_\alpha$. Similarly denote by $\bb_y$ the infinite sequence over $\{2,3\}$ coding the point $y\in[0,1)$ under $R_\gamma$.
% We also denote by $\vv_(x,y)$ the infinite sequence coding the point $(x,y)$ under the transformation $S$.
Recall the partition of the interval $[0,1)$ into subintervals with the same coding under two interval exchange transformation. 
For a given factor $w\in\LL(\uu)$ of length $n$, denote
$$
I_w = \{x\in[0,1) : \text{ the sequence $x$, $T_\alpha(x)$, $T_\alpha^2(x)$, \dots, $T^{n-1}_\alpha(x)$ is coded by } w\}.
$$
For a factor $z\in\LL(\bb)$ of length $k=|w|_b$, we define
$$
J_z = \{y\in[0,1) : \text{ the sequence $y$, $T_\gamma(y)$, $T_\gamma^2(x)$, \dots, $T^{k-1}_\gamma(x)$ is coded by } z\}.
$$
%It follows from the properties of interval exchange transformations that  the cylinders are intervals, and thus the set $I_w\times J_z\subset\Omega$ has a non=empty interior. 
Since, by assumption, 1, $\alpha$ and $\alpha\gamma$ are rationally independent, we use Proposition~\ref{p:huste} to derive that there exists an index $m\in\N$ such that $S^m(0,0)\in I_w\times J_z$. Denote $S^m(0,0)=(x_0,y_0)$.
It is straightforward to verify using~\eqref{eq:RectTrans} that 
the desired factor $v$ can be found as the symbolic coding of 
$(x_0,y_0)$, $S(x_0,y_0)$, $S^2(x_0,y_0)$, \dots, $S^{n-1}(x_0,y_0)$.

Consider now abelian complexity. The number of occurrences of letter $b$ in factors of $\uu$ of length $n$ is either $\lfloor n\alpha\rfloor$, or $\lceil n\alpha \rceil$ (both cases occur). 
It implies that for $n \geq \lceil \frac{1}{\alpha}\rceil$, each factor of length $n$ in $\uu$ contains at least one $b$. 
There are two distinct Parikh vectors for factors of length $n$ in $\uu$. After colouring, each of them gives rise to two distinct Parikh vectors in $\vv$. 
Consequently, $\mathcal{C}^{ab}_\vv (n) = 4$ for $n \geq \lceil \frac{1}{\alpha}\rceil$.
For $1\leq n\leq\lfloor \frac{1}{\alpha}\rfloor$, each factor of length $n$ in $\uu$ contains either no letter $b$, or one letter $b$. The first one gives rise to one Parikh vector in $\vv$, the other one to two Parikh vectors. Therefore $\mathcal{C}^{ab}_\vv (n) =3$ for $1\leq n \leq\lfloor \frac{1}{\alpha}\rfloor$.
\end{proof}

\section{Symbolic computation of factor complexity} \label{sec:numeration_systems}

When determining factor complexity in Theorem~\ref{thm:ternary}, we need to compute the values of the function $P_\alpha$.  
They may be calculated using the formula from Lemma \ref{funkceP}. A~disadvantage is the fact that one has to determine the fractional part of $k\alpha$ for all $k\leq n$ with sufficient precision.

In the sequel, for a class of irrational numbers $\alpha \in (0,1)$, we will show how to compute the values $P_\alpha(n)$ in a symbolic way using the greedy representation of integers in a non-standard numeration system associated with $\alpha$. More precisely, we will consider $\alpha$ such that the number $\beta = \frac{1}{\alpha}$ is a quadratic non-simple Parry unit, i.e., 
\begin{equation}\label{eq:beta}
    \text{$\beta=\frac{m+\sqrt{m^2-4}}{2}$ is equal to the larger root of the polynomial $x^2-mx+1$,}
    \end{equation}
     where $m\in \mathbb N, m \geq 3$.
    We associate with $\beta$ the sequence 
 ${\bf U}=\bigl(U_k\bigr)_{k\in \N}$ defined by recurrence relation  
 \begin{equation}\label{U}
 U_{-1} = 0,\ \ U_0 = 1,\quad \text{and}  \quad  U_{k+1} = mU_{k}-U_{k-1} \ \ \text{ for every $k \in \N$.}
 \end{equation} 
Now for $\beta$, we define the $\beta$-expansion of real numbers from $[0,1)$,  and  for the sequence  ${\bf U}$, we define the ${\bf U}$-expansion of non-negative integers.  For more details on  both types of expansion see Chapter 7.3 by Frougny~\citep{Fr02}.
\bigskip

A $\beta$-representation of a number $x \in [0,1)$  is an infinite sequence $(x_i)_{i\geq 1}$ of non-negative integers such that $x = \sum_{i=1}^{+\infty} x_i\beta^{-i}$.  A particular $\beta$-representation -- called the  $\beta$-expansion  -- can be computed by the greedy algorithm which uses a transformation $T_\beta$ defined  by $T_\beta(x) =  \beta x - \lfloor \beta \,x \rfloor \ \text{for every } x \in [0,1). $

If we put   $a_i = \lfloor \beta \,T^{i-1}(x) \rfloor$ for every $i \in \N , i\geq 1$,  then $x = \sum_{i=1}^{+\infty}  a_i\beta^{-i}$  and $a_i  < \beta$  for every $i\geq 1$. The string $(a_i)_{i\geq 1}$ is denoted by $d_\beta(x)$. The function $d_\beta(x)$ is increasing with respect to the lexicographical order.   In the space of  infinite sequences over $\{ a \in \N: a< \beta\}$ equipped with the product topology, we can define  the quasi-greedy expansion of 1 as $d^*_\beta(1)  = \lim\limits_{x\to 1-}d_{\beta}(x)$.   

As proven by Parry~\citep{Pa1960}, 
a sequence of non-negative integers  $(x_i)_{i\geq 1}$  is the $\beta$-expansion of a~number $x \in [0,1)$ if and only if \begin{equation}\label{lex}
x_ix_{i+i}x_{i+2}\cdots\prec_{lex} d^*_\beta(1) \ \ \text{for every } \ i \in \N, i \geq 1.\end{equation} It is easy to verify that for $\beta =\frac{m+\sqrt{m^2-4}}{2}$, the quasi-greedy expansion of 1 is $d^*_\beta(1) = (m-1)(m-2)^\omega$.  

\begin{example}\label{ex:goldenmean}
For $m=3$, $\beta$ is the larger root of $x^2-3x+1=0$, i.e., $\beta=\tau^2$, where $\tau=\frac{1+\sqrt{5}}{2}$ is the golden mean, and $d^*_\beta(1) = 21^\omega$.
For $m=4$, $\beta$ is the larger root of $x^2-4x+1=0$, i.e., $\beta=2+\sqrt{3}$, and $d^*_\beta(1) = 32^\omega$.
\end{example}

\medskip

A ${\bf U}$-representation of a number  $n\in \N$ is a finite sequence of non-negative integers  $a_Na_{N-1}\cdots a_0$ such that  $ n=\sum_{k=0}^N a_kU_k$. If moreover,  $a_N \neq 0$ and $\sum_{k=0}^i a_k U_k < U_{i+1}$ for every $i \leq N$, then the representation is called the ${\bf U}$-expansion of $n$ and will be denoted 
$$(n)_{\bf U} =a_Na_{N-1}\cdots a_1a_0 \,. $$
Notice that $n=0$ does not have any representation with a non-zero leading coefficient, therefore we define its ${\bf U}$-expansion to be the empty word.

The ${\bf U}$-expansion of $n$  can be computed by the greedy algorithm. Our sequence ${\bf U}$ satisfies for  every $\ell, k \in \N, k\geq 1$,  the inequalities $$mU_\ell  \geq   U_{\ell+1}\quad \text{and} \quad U_{\ell+k+1} + U_{\ell-1} \geq U_{\ell+k+1}.$$ Note that
$$(m-1)U_{\ell+k} + (m-2)U_{\ell+k-1}+\ldots +(m-2)U_{\ell+1} + (m-1)U_{\ell}=U_{\ell+k+1} + U_{\ell-1}.$$
\noindent Therefore the finite sequence $(n)_{\bf U}$ is 
\begin{equation}\label{Uexp}
\text{a word  over $\{0,1,\dots, m-1\}$ not containing  $(m-1)(m-2)^{k-1}(m-1)$  as factor.}
\end{equation}
 In fact, the above property under the assumption that the leading coefficient is non-zero characterizes the ${\bf U}$-expansion  of  integers. 
\begin{proposition}\label{fraction} Let  ${\bf U}=\bigl(U_k\bigr)_{k\in \N}$  be the sequence defined in \eqref{U} and $\beta$ be the number defined in~\eqref{eq:beta}. Let $n \in \N$  be an integer with   $(n)_{\bf U} =a_Na_{N-1}\cdots a_1a_0 $. Then the ${\bf U}$-expansion of the integer part  of  $ \tfrac{n}{\beta}$ and the $\beta$-expansion of the fractional part of  $ \tfrac{n}{\beta}$ are  
$$
\Bigl(\lfloor \tfrac{n}{\beta} \rfloor \Bigr)_{\bf U}  = a_Na_{N-1}\cdots a_1\qquad \text{and} \qquad  d_\beta \Bigl(\bigr\{ \tfrac{n}{\beta} \bigr\}\Bigl) = a_0a_1\cdots a_{N-1}a_N 0^\omega. 
  $$ 
\end{proposition}
\begin{proof}  
If $n<m$, then $(n)_{\bf U}=a_0=n$ and $\lfloor \tfrac{n}{\beta} \rfloor=0$ and $\bigr\{ \tfrac{n}{\beta} \bigr\}=\frac{n}{\beta}$. Thus the statement holds.
In the rest of the proof, assume $n\geq m$.
Let us prove by induction the formula 
\begin{equation}\label{eq:induction}    
\tfrac{1}{\beta}U_k= U_{k-1} + \beta^{-(k+1)}
\end{equation}
for all $k\in \N$. 
For $k=0$, we have $\tfrac{1}{\beta}U_0= U_{-1} + \beta^{-1}$. Since $U_{-1}=0$ and $U_0=1$, it holds. For $k=1$, we have $\tfrac{1}{\beta}U_1= U_{0} + \beta^{-2}$. Inserting $U_1=m$ and multiplying by $\beta^2$ leads to $m\beta= \beta^2+ 1$, which is satisfied by definition of $\beta$~\eqref{eq:beta}. Assume the equality~\eqref{eq:induction} holds for all $k\leq j$, where $j\geq 1$ is a~fixed integer. Let us prove that then also
$\tfrac{1}{\beta}U_{j+1}= U_{j} + \beta^{-(j+2)}$. 
$$\begin{array}{rcl}
\tfrac{1}{\beta}U_{j+1}&=&\tfrac{1}{\beta}(mU_{j}-U_{j-1})\\
&=&m(U_{j-1}+\beta^{-(j+1)})-(U_{j-2}+\beta^{-j})\\
&=&mU_{j-1}-U_{j-2}+\frac{m\beta-\beta^2}{\beta^{j+2}}\\
&=&U_j+\beta^{-(j+2)},
\end{array}$$
where we used the recurrence relation~\eqref{U} in the first equality, the induction assumption in the second equality, and relations \eqref{U} and \eqref{eq:beta} in the last equality.

Hence 
      $$\frac{n}{\beta} = \frac{1}{\beta}\sum_{k=0}^{N}a_kU_k  = \sum_{k=0}^Na_kU_{k-1} +  \sum_{k=0}^Na_k \beta^{-(k+1)}  = \sum_{k=1}^Na_kU_{k-1} + \sum_{k=1}^{N+1}a_{k-1} \beta^{-k}\,.$$
      As the finite string $a_Na_{N-1}\cdots a_1a_0$ satisfies \eqref{Uexp},  the infinite string  $a_0a_1a_2\cdots a_N 0^\omega $ fulfills  \eqref{lex} with  $d^*_\beta(1) = (m-1)(m-2)^\omega$. It implies that  $a_0a_1a_2\cdots a_N 0^\omega $ is the $\beta$-expansion of a number  $x \in [0,1)$.  Together with the fact that $ \sum_{k=1}^Na_kU_{k-1} $ is an integer and the finite string  $a_Na_{N-1}\cdots a_1$ satisfies \eqref{Uexp} as well, it gives both parts of the statement.  
\end{proof}

\begin{theorem}\label{formule} Let  ${\bf U}=\bigl(U_k\bigr)_{k\in \N}$  be the sequence defined in \eqref{U} and $\alpha = \frac{1}{\beta}$ for $\beta$ given in~\eqref{eq:beta}. If the ${\bf U}$-expansion of a number $n \in \N$ is   $(n)_{\bf U} =a_Na_{N-1}\cdots a_1a_0 $, then 
$$
P_\alpha(n) = a_{N}\bigl(r_N+1\bigr) + a_{N-1}\bigl(r_{N-1}+1\bigr)+ \cdots + a_{1}\bigl(r_1+1\bigr) + a_0, 
$$
where $r_i$  is an integer having ${\bf U}$-representation $a_0a_1\cdots a_{i-1}$ for every $i \in \N, 1\leq i\leq N$.  
\end{theorem}

\begin{proof}   We say that a finite string  $b_jb_{j-1}\cdots b_1b_0$    is admissible if it satisfies \eqref{Uexp}. In other words, if  we erase the leading zeros (if any) in  $b_jb_{j-1}\cdots b_1b_0$,    we get the ${\bf U}$-expansion of some integer.  
By~\eqref{Uexp},   a~string $B:= b_jb_{j-1}\cdots b_1b_0$
 is admissible  if and only if its reverse   $\overline{B}:=b_0b_1\cdots b_{j-1}b_j$ is admissible.

Let us fix  $k\in \N, 1\leq k\leq n$. Such $k$  has exactly one admissible ${\bf U}$-representation of length $N+1$, say $b_Nb_{N-1}\cdots b_1b_0$.  The greediness of the  ${\bf U}$-expansion guarantees that 
\begin{equation}\label{nerovnost}b_Nb_{N-1}\cdots b_1 b_0 \preceq_{lex} a_Na_{N-1}\cdots a_1a_0\,. \end{equation}
By Proposition \ref{fraction}, the $\beta$-expansion of the fractional part of $\frac{k}{\beta}$ is $b_0b_1\ldots b_N0^\omega$. By Lemma \ref{funkceP}, such $k$  contributes to $P_\alpha(n)$ if and only if the fractional parts satisfy $\{ \frac{k}{\beta}\} \leq \{ \frac{n}{\beta}\}$. The greediness of the $\beta$-expansion and Proposition \ref{fraction}  allow us to reformulate the inequality  into  
\begin{equation}\label{revers}
b_0b_1\cdots b_{N-1}b_N \preceq_{lex} a_0a_1\cdots a_{N-1}a_N \,.
\end{equation}
Denote $$
M=\{b_Nb_{N-1}\cdots b_0\ :\  b_Nb_{N-1}\cdots b_0 \  \text{admissible  and satisfies \  \eqref{nerovnost} \ and \eqref{revers}} \} \,,\qquad 
$$
then $P_\alpha(n)=\# M-1$. We subtract one because the zero sequence represents $k=0$.

Let us fix $i \in \N$,    $1\leq i \leq N$, and consider an admissible sequence   $b_{0}b_{1}\cdots b_{i-1}$ such that 
\begin{equation}\label{fixi}b_{0}b_{1}\cdots b_{i-1}\preceq_{lex} a_0a_1\ldots a_{i-1}. \end{equation}
Then the sequence $B:=a_Na_{N-1}\cdots a_{i+1}x b_{i-1}\cdots b_1 b_0$ with $x \in \N, \ x < a_{i}$, is admissible as well,   $B\prec_{lex}  a_Na_{N-1}\cdots a_1a_0 $ and $\overline{B} \prec_{lex}  a_0a_1\cdots a_{N-1}a_N $. 
In the statement of the theorem, $r_i$ denotes the non-negative number with  a ${\bf U}$-representation $a_0a_1\ldots a_{i-1}$. Thus there are  $r_i+1$ non-negative integers smaller than or equal to   $r_i$ and consequently  $r_i+1$ admissible strings  $b_{0}b_{1}\cdots b_{i-1}$  satisfying  \eqref{fixi}.  Moreover, we have  $a_i$ possibilities to choose a non-negative integer  $x < a_i$. 

For every  fixed $i \in \N$,    $1\leq i \leq N$,  we have found $a_i(r_i+1)$ sequences of the form $B:=a_Na_{N-1}\cdots a_{i+1}x b_{i-1}\cdots b_1 b_0$ which belong to $M$.    Moreover, $M$ contains  also $a_0+1$ sequences of the form  $B:=a_Na_{N-1}\cdots a_1 x$ with $x\leq a_0$. It is easy to see that any element of $M$ is of the form $B$ we considered. 
\end{proof}

\begin{example} First, let us illustrate Theorem~\ref{formule} for $\alpha=\frac{1}{\tau^2}$. This corresponds to the unique case  we studied already in the conference paper~\citep{DMP2023}. See Example~\ref{ex:goldenmean}. According to~\eqref{U} we have $U_0=1, U_1=3, U_2=8, U_3=21$. Let $n=24$. The  ${\bf U}$-expansion of $n$ is
$(n)_{\bf U} = 1010$ as $24 = 1\cdot U_3+ 0\cdot U_2 +1\cdot U_1+0\cdot U_0$.
For $(n)_{\bf U}  = a_3a_2a_1a_0 = 1010$, we first determine $(r_3)_{\bf U} = 10$,  $(r_2)_{\bf U} = 1$, $(r_1)_{\bf U} = 0$, i.e.,  $r_3=3$, $r_2 = 1$, $r_1 = 0$. By the formula from Theorem \ref{formule},  $P_\alpha(24)= 1\cdot 4 + 0\cdot 2 + 1\cdot 1 +0= 5$. 

In Table \ref{tablePU}, there are the values of $P_\alpha(n)$ for $n \leq 29$.
Observe that $P_\alpha(9) = P_\alpha(24) = 5$. The ${\bf U}$-expansion of $9$ is $(9)_{\bf U} = 101$, i.e., it is the  mirror image of the ${\bf U}$-expansion of $n = 24$.  The equality  of  $P_{\alpha}(9)$ and $P_{\alpha}(24)$ is a consequence of the symmetry of inequalities \eqref{nerovnost} and \eqref{revers}.  

Second, let us consider $\alpha=2-\sqrt 3$. See the second part of Example~\ref{ex:goldenmean}. According to~\eqref{U} we have $U_0=1, U_1=4, U_2=15, U_3=56$. Let $n=24$. The  ${\bf U}$-expansion of $n$ is
$(n)_{\bf U} = 121$ as $24 = 1\cdot U_2+ 2\cdot U_1 +1\cdot U_0$.
For $(n)_{\bf U}  = a_2a_1a_0 = 121$, we first determine $(r_2)_{\bf U} = 12$,  $(r_1)_{\bf U} = 1$, i.e.,  $r_2=6$, $r_1 = 1$. By the formula from Theorem \ref{formule},  $P_\alpha(24)= 1\cdot 7 + 2\cdot 2 + 1= 12$.
See Table~\ref{PU2}.

\end{example}
\begin{table}[ht]
\renewcommand{\arraystretch}{1.2}
\setlength{\tabcolsep}{5pt}
\begin{tabular}{|c|c|c|c|c|c|c|c|c|c|c|c|c|c|c|c|c|}
    \hline 
    $n$ & ~1 &~2 & 3 &  4 &  5 &  6 &  7 &  8 &  9 & 10 & 11 & 12 & 13 & 14 & 15 & 16 \\ 

     \hline
   $(n)_{\bf U}$ & ~1 & ~2 & 10 &  11 &  12 &  20 &  21 &  100 &  101 & 102 & 110 & 111 & 112 & 120 & 121 & 200 \\ 

     \hline
     
    $P_{\alpha}(n)$ & ~1 & ~2 & 1 & 3 & 5 & 2 & 5 & 1 & 5 & 9 & 3 & 8 & 13 & 5 & 11 & 2\\
    \hline
\end{tabular}

\medskip

\begin{tabular}{|c|c|c|c|c|c|c|c|c|c|c|c|c|c|}
    \hline 
    $n$ & 17 & 18 & 19 & 20  & 21& 22 & 23 & 24 & 25 & 26 & 27 & 28 & 29    \\ 
    \hline
       
   $(n)_{\bf U}$  & 201 & 202 & 210 & 211  & 1000 & 1001 & 1002 &  1010 &  1011 &  1012 &  1020 &  1021 &  1100  \\ 
         \hline
    $P_{\alpha}(n)$ & 9 & 16 & 5 & 13    &1&  10 & 19 & 5 & 15 & 25 & 9 & 20 & 3    \\
    \hline
\end{tabular}

\caption{$\bf U$-expansion of $n \leq 29$ and $P_\alpha(n)$  for  $\alpha=\frac{1}{\tau^2}$ }\label{tablePU}

\end{table}

\begin{table}[ht]
\renewcommand{\arraystretch}{1.2}

\begin{tabular}{|c|c|c|c|c|c|c|c|c|c|c|c|c|c|c|c|c|}
    \hline 
    $n$ & ~1 &~2 & 3 &  4 &  5 &  6 &  7 &  8 &  9 & 10 & 11 & 12 & 13 & 14 & 15 & 16 \\ 

     \hline
   $(n)_{\bf U}$ & ~1 & ~2 & 3 &  10 &  11 &  12 &  13 &  20 &  21 & 22 & 23 & 30 & 31 & 32 & 33 & 101 \\ 

     \hline
     
    $P_{\alpha}(n)$ & ~1 & ~2 & 3 & 1 & 3 & 5 & 7 & 2 & 5 & 8 & 11 & 3 & 7 & 11 & 1 & 6\\
    \hline
\end{tabular}
\medskip

\begin{tabular}{|c|c|c|c|c|c|c|c|c|c|c|c|c|c|}
    \hline 
    $n$ & 17 & 18 & 19 & 20  & 21& 22 & 23 & 24 & 25 & 26 & 27 & 28 & 29    \\ 
    \hline
       
   $(n)_{\bf U}$  & 102 & 103 & 110 & 111  & 112 & 112 & 120 &  121 &  122 &  123 &  130 &  131 &  132  \\ 
         \hline
    $P_{\alpha}(n)$ & 11 & 16 & 3 & 9    &15&  21 & 5 & 12 & 19 & 26 & 7 & 15 & 23    \\
    \hline
\end{tabular}

\caption{$\bf U$-expansion of $n \leq 29$ and $P_\alpha(n)$  for  $\alpha=2-\sqrt{3}$}\label{PU2}

\end{table}

If we want to compute the values of the function $P_\alpha(n)$  for all consecutive integers $n = 1,2,\ldots, N$, it is more convenient to work with the difference $\Delta P_\alpha(n) = P_\alpha(n+1)-P_\alpha(n)$. The following formulae may be deduced easily  from  Theorem~\ref{formule}.

\begin{corollary}  Let $(n)_{\bf U} =a_Na_{N-1}\cdots a_1a_0 $ and    $b_0 $ be the last digit in the ${\bf U}$-expansion of the number $n+1$. 
\begin{itemize}\item If $b_0 =0$, then  $a_Na_{N-1}\cdots a_1 $ is the ${\bf U}$-expansion of 
$n +\Delta P_\alpha(n) -1$.  
\item If $b_0 \neq 0$, then  $a_Na_{N-1}\cdots a_1 $ is the ${\bf U}$-expansion of 
$\Delta P_\alpha(n) -1$. 
 
\end{itemize}
\end{corollary}

\section{Comments}
\begin{itemize}

% \item In computer experiments,  
%  we counted in a prefix of $\vv$  of length $500 000$  the number of different factors of length $n$ for $n=1,2,\ldots, 300$. Under the assumption in Conjecture~\ref{conj:ternary},  
%   the values we found out coincide with  the upper bound. 
% The Sturmian sequences we used in our experiments are fixed points of Sturmian morphisms, in particular, we limited our consideration to frequencies being quadratic numbers.
\item Regarding abelian complexity, Richomme, Saari, and Zamboni~\citep{RiSaZa2010} showed that ternary aperiodic 1-balanced sequences have constant abelian complexity equal to 3. Currie and Rampersad~\citep{CuRa11} showed that there are no recurrent sequences of constant abelian complexity equal to 4. Saarela~\citep{Sa09} proved for every integer $c\geq 2$ that there exists a recurrent sequence with abelian complexity ${\mathcal C}^{ab}(n)=c$ for all $n \geq c-1$. Using Theorem~\ref{thm:ternary} with  $\alpha>1/2$, we provide infinitely many ternary 2-balanced sequences $\vv$ of this type, more precisely, such that  ${\mathcal C}_{\vv}^{ab}(n)=4$ for each $n \geq 2$.

\item
We believe that the method derived for computing $P_{\alpha}(n)$ may be generalized to all cases, where $1/\alpha$ is a~quadratic unit (not only non-simple Parry unit). For other numbers $\alpha$, the Ostrowski numeration system or its modification could be helpful. 
\item Our future goal is to study combinatorial properties of ternary 2-balanced sequences from the new class more deeply, in particular, we aim to describe bispecial factors and return words with the intention of determining critical exponents of these sequences. 
\item The ternary 2-balanced sequences studied in this paper are given by two parameters -- slopes of the coloured and colouring Sturmian sequence. Similarly, ternary billiard sequences are 2-balanced and are also defined by two parameters -- slope of a line in $\mathbb R^3$. Even the abelian complexity, for almost all pairs of parameters, is equal to $4$ for sufficiently large $n$. Still, these two classes are different. While the factor complexity of ternary billiard sequences, for almost all pairs of parameters, equals $n^2+n+1$, the factor complexity of ternary 2-balanced sequences from our class is upper bounded by $\alpha n^2\bigl(1+ o(1)\bigr)$, where $\alpha\in (0,1)$ is the slope of the coloured Sturmian sequence (see Remark~\ref{rem:asymptotic behavior}).

\end{itemize}

%%=======================sn====================================================================%%

\bibliography{sn-bibliography}

\end{document}